\documentclass[12pt,a4paper]{amsart}
\usepackage[utf8]{inputenc}
\usepackage[T1]{fontenc}
\usepackage{amsmath,amsfonts,amssymb,amsthm,mathrsfs}
\usepackage[normalem]{ulem}%strikeout \sout
\usepackage{color}
\usepackage{marginnote}

\newcommand{\be}{\begin{equation}} 
\newcommand{\ee}{\end{equation}}
\newcommand{\bea}{\begin{eqnarray}} 
\newcommand{\eea}{\end{eqnarray}}
\newcommand{\bean}{\begin{eqnarray*}} 
\newcommand{\eean}{\end{eqnarray*}}

\usepackage{hyperref}
\hypersetup{
  colorlinks=true,
  linkcolor=black,
  citecolor=black,
  urlcolor=blue,
  pdftitle={Asymptotics of a chemotaxis-consumption-growth model},
  pdfauthor={Knosalla, Lankeit},
  pdfkeywords={},
  bookmarksopen=false,
}

\makeatletter
\@namedef{subjclassname@2020}{\textup{2020} Mathematics Subject Classification}
\makeatother

\def\mn{|\!\!|}

\def\mn2{|\!\!|_{M^{d/2}}}

\usepackage{newunicodechar}
\newunicodechar{σ}{\sigma}
\newunicodechar{ℝ}{\mathbb{R}}
\newunicodechar{ℕ}{\mathbb{N}}
\newunicodechar{λ}{\lambda}
\newunicodechar{μ}{\mu}
\newunicodechar{τ}{\tau}
\newunicodechar{ξ}{\xi}
\newunicodechar{ε}{\varepsilon}
\newunicodechar{∇}{\nabla}
\newunicodechar{θ}{\theta}
\newunicodechar{∞}{\infty}
\newunicodechar{γ}{\gamma}
\newunicodechar{Δ}{\Delta}

\newcommand{\io}{\int_\Omega}
\newcommand{\intdom}{\int_{\partial\Omega}}
\newcommand{\vbar}{\overline{v}}
\newcommand{\Lom}[1]{L^{#1}(\Omega)}
\newcommand{\Tmax}{T_{max}}
\newcommand{\Ombar}{\overline{\Omega}}
\newcommand{\utilde}{\widetilde{u}}
\newcommand{\vtilde}{\widetilde{v}}
\newcommand{\vhat}{\hat{v}}

\newtheorem{theorem}{Theorem}

\newtheorem{lemma}[theorem]{Lemma}

\theoremstyle{definition}

\theoremstyle{remark}
\newtheorem{remark}[theorem]{Remark}

\numberwithin{equation}{section}
\numberwithin{theorem}{section}

\author[P. Knosalla]{Piotr Knosalla}
\address[P. Knosalla]{
Institute of Physics, Opole University, Oleska 48, 45-052 Opole, Poland\\
ORCID: 0000-0002-3594-0938}
\email{piotr.knosalla@uni.opole.pl}

\author[J. Lankeit]{Johannes Lankeit}
\address[J. Lankeit]{Leibniz Universität Hannover, Institut für Angewandte Mathematik, Welfengarten 1, 30167 Hannover, Germany\\
ORCID: 0000-0002-2563-7759}
\email{lankeit@ifam.uni-hannover.de}

\thanks{}

\title[Asymptotics of a chemotaxis-consumption-growth model]{
Asymptotics of a chemotaxis-consumption-growth model with nonzero Dirichlet conditions}

\newcommand\norm[1]{\left\lVert#1\right\rVert}

\newcommand{\normA}[2][]{\|#2\|_{#1}} 
\newcommand{\Vnorm}[2][t_1,t]{\normA[V,#1]{#2}}

\allowdisplaybreaks

\begin{document}
%%%%%%%%%%%%%%%%%%%%%%%%%%%%%%%%%%%%%%%%%%%%%%%%%%%%%%%%%%%%%%%%%%%%%%
\begin{abstract}
This paper concerns the asymptotics of certain parabolic-elliptic chemotaxis-consumption systems with logistic growth and constant concentration of chemoattractant on the boundary. First we prove that in two dimensional bounded domains there exists a unique global classical solution which is uniformly bounded in time, and then we show that if the concentration of chemoattractant on the boundary is sufficiently low then the solution converges to the positive steady state as time goes to infinity.
\end{abstract}


%35B40  Asymptotic behavior of solutions to PDEs 
%35A01  Existence problems for PDEs: global existence, local existence, non-existence 
%35K20  Initial-boundary value problems for second-order parabolic equations

\date{\today}
%%%%%%%%%%%%%%%%%%%%%%%%%%%%%%%%%%%%%%
\maketitle

\baselineskip=16pt

%%%%%%%%%%%%%%%%%%%%%%%%%%%%%%%%%%%%%%%%%%%%%%%%%%%%%%%%%%%%%%%%%%%%%%%%%%%%%%%%%%%%%%%%%%%%%%%%%%%%
% Main matter
%%%%%%%%%%%%%%%%%%%%%%%%%%%%%%%%%%%%%%%%%%%%%%%%%%%%%%%%%%%%%%%%%%%%%%%%%%%%%%%%%%%%%%%%%%%%%%%%%%%%
\section{Introduction}

    We denote by $u=u(x,t)$ the density of a colony of bacteria living in a bounded domain $\Omega\subset \mathbb{R}^n$ filled with water with $C^{2,\sigma}$-boundary for some $\sigma\in(0,1)$, and by $v=v(x,t)$ the density of chemoattractant which is dissolved in the water, $x\in \Omega,$  $t>0$. Bacteria are random walking and moving towards higher concentration of chemoattractant (see papers by Keller and Segel, \cite{KeSe, KeSe1, KeSe2}). The growth of the colony is logistic. We assume that the concentration of chemoattractant is positive and constant along the boundary $\partial \Omega$ of the domain. We further assume that $v$ is diffusing and is consumed by the bacteria. Neglecting the effect of the fluid motion and assuming that the diffusion of $v$ is very high, we end up with the following parabolic-elliptic system describing the evolution of the model,   
\begin{equation}\label{systPE}
\left\{ \begin{array}{l}
 u_t=\nabla\cdot(\nabla u-u\nabla v)+\lambda u-\mu u^2, \\
 \partial_\nu u-u\partial_\nu v|_{\partial\Omega}=0,\\
 -\Delta v=-uv,\\
 v|_{\partial\Omega}=\overline{v},\\
  u(x,0)=u_0(x).
 \end{array} \right.
\end{equation}
Here $\lambda,$ $\mu,$ $\overline{v}$ are positive constants, and $\nu=\nu(x)$ is the outward normal vector to the boundary at point $x\in\partial \Omega$. The initial bacteria density $u_0$ is assumed to be a positive and continuous function. This paper is devoted to the question of global existence and asymptotic behaviour of solutions of this system.

System \eqref{systPE} is a simplification of the following parabolic-parabolic system 
\begin{equation}\label{systPP}
\left\{ \begin{array}{l}
 u_t=\nabla\cdot(\nabla u-u\nabla v)+\lambda u-\mu u^2, \\
 v_t=\Delta v-vu,\\
  u(x,0)=u_0(x),\quad v(x,0)=v_0(x).
 \end{array} \right.
\end{equation} 
which can be considered with Neumann, Dirichlet or Robin boundary condition imposed on the chemoattractant $v$ and no-flux boundary condition on $u$. Of course, the type of boundary conditions imposed on $v$ describes different physical situations and may lead to serious obstructions and challenges in the mathematical analysis of particular system. Global existence and convergence to the constant steady state $(\frac{\lambda}{\mu},0)$ of solutions of system \eqref{systPP}, with homogeneous Neumann boundary condition, were considered by Lankeit and Wang in \cite{LankeitWang}. The steady state problem of this system with Dirichlet boundary condition was studied by Knosalla and Wróbel, \cite{KW}. The most popular and best analyzed in the literature is the growthless system

\begin{equation}\label{systPP:ng}
\left\{ \begin{array}{l}
 u_t=\nabla\cdot(\nabla u-u\nabla v), \\
 v_t=\Delta v-vu,\\
  (\partial_\nu u-u\partial_\nu v)|_{\partial\Omega}=0,\\
  u(x,0)=u_0(x),\quad v(x,0)=v_0(x)
 \end{array} \right.
\end{equation}
with $\partial_{\nu}v|_{\partial\Omega}=0. $

 Systems \eqref{systPP}, \eqref{systPP:ng} are the fluid-free versions of the famous chemotaxis-consumption-fluid model proposed by Tuval et al. in \cite{Tuval}. The existence of global solutions of the chemotaxis-fluid model without growth and under zero Neumann boundary condition on $v$ was studied by Winkler in \cite{Winkler}. Braukhoff and Tang, \cite{BrTang}, and Braukhoff, \cite{Braukhoff}, considered system \eqref{systPP:ng}, \eqref{systPP},  respectively, with fluid and Robin boundary condition on $v$.  System \eqref{systPP} with fluid and Dirichlet boundary condition on $v$ was analyzed by Black and Winkler, \cite{BlackWinkler}, Black and Wu, \cite{BlackWu, BlackWu1}. Wang, Winkler and Xiang, \cite{WangWinklerXiang1, WangWinklerXiang2} considered global solvabilty of system \eqref{systPP:ng} with fluid and Dirichlet boundary condition. \\
Tao and Winkler in \cite{TWin} showed the global existence of solutions of system \eqref{systPP:ng}  and its convergence to the unique constant steady state $(\frac{1}{|\Omega|}\int_{\Omega}u_0,0)$. In \cite{BrLank}, Braukhoff and Lankeit obtained existence and uniqueness of non-constant stationary solutions of \eqref{systPP:ng} with inhomogeneous Robin boundary condition on $v$, and in \cite{FLM} Fuest et al.\ showed global existence of a solution of the parabolic-elliptic counterpart of \eqref{systPP:ng} and its convergence to this steady state. 
The global solvability of the parabolic-elliptic system with tensor-valued sensitivities was shown by Ahn, Kang and Lee in \cite{AhnKangLee}; with nonzero Dirichlet boundary conditions for the signal and general sensitivity functions, a proof of global existence and, in special cases, convergence can be found in the work \cite{YangAhn} by Yang and Ahn.
Existence, uniqueness and boundary layer profile of solutions of the stationary problem of \eqref{systPP:ng} with Dirichlet boundary condition were studied by Lee et al. in \cite{LeeWangYang}. Hong et al. in \cite{HongWang} considered asymptotic behaviour of system \eqref{systPP:ng} in one space dimension with Dirichlet boundary condition. In \cite{LankeitWinkler} Lankeit and Winkler obtained the global existence of solutions of this system under the assumption of radial symmetry. Recently Li and Li in \cite{LiLi} showed local asymptotic stability of steady state solution of \eqref{systPP:ng} with non-zero Dirichlet or Robin boundary condition. Unique solvability of the stationary problem of a system with logarithmic sensitivity was proven by Ahn and Lankeit \cite{AhnLankeit}.
It is also worth pointing out the papers by Knosalla and Nadzieja, \cite{KNadz}, and Knosalla, \cite{KNos, Knos1, Knos2},   
which concern the aerotaxis model. Evolution of the model considered therein is described by the system \eqref{systPP:ng} with appropriately chosen sensitivity function, called the {\it energy}, and with inhomogeneous Neumann or Dirichlet boundary condition imposed on  $v$. For more references concerning systems like \eqref{systPP:ng}, \eqref{systPP} see the recent overviewing paper by Lankeit and Winkler, \cite{LankeitWinklerD}.

\section{Main results}
In this section we present the main results of this paper. Assumptions that we will pose regularly are the following: 
\begin{subequations}\label{assumptions}
\begin{align}
n&=2\label{ass:dimension}\\
\Omega&\subset ℝ^n \text{ is a bounded domain with } C^{2,σ}\text{-boundary, } σ\in(0,1)\label{ass:domain}\\
μ&>0,λ\ge 0, m>0\label{ass:params}\\
u_0&\in W^{1,p_0}(\Omega),\; p_0>n,\;\; u_0 > 0 \text{ in }\Ombar,\;\; \io u_0=m\label{ass:init}
\end{align}
\end{subequations}
Solutions to \eqref{systPE} additionally depend on the remaining parameter $\vbar>0$. However, in a proof concerning long-time behaviour in Section~\ref{sec:asymptotics}, it will be crucial that a certain bound remains independent of $\vbar$, and we hence indicate assumptions on $\vbar$ separately.
(Note that ``Assume \eqref{assumptions}. There is $C>0$ such that for all $\vbar>0$ \ldots'' admits dependence of $C$ on $\Omega$, $μ$, $λ$, $u_0$, but not on $\vbar$.) \\

We begin with local existence of solutions, namely using the semigroup approach due to Amann \cite{Amann1}, we show our first result. 
\begin{theorem}[Local existence]\label{locEx}
Let $n\in ℕ$, assume \eqref{ass:domain}, \eqref{ass:params}, \eqref{ass:init}. Then for every $\vbar>0$, 
there exists a unique local solution 
$$
u\in C([0,T_{\max} ); W^{1,p}(\Omega))\cap C^{2,1}(\overline{\Omega}\times(0,T_{\max})),
$$
$$
v\in C([0,T_{\max});C^{2,\sigma_*}(\overline{\Omega}))\cap C^{0,1}((0,T_{\max});C^{2,\sigma_*}(\overline{\Omega})) 
$$
 for some $\sigma_*\in(0,1)$ of \eqref{systPE}, where $T_{\max}\in(0,\infty]$ is the maximal time of existence of the solution. This solution satisfies $u(x,t)>0$ and $\overline{v}> v(x,t)>0$ for $(x,t)\in\overline{\Omega}\times(0,T_{\max}).$ Moreover, if for every $T>0$ there exists a constant $C(T)$ such that  
\begin{equation}\label{cond:globalex}
\sup_{t\in[0,T]\cap [0,T_{\max})}\normA[L^\infty(\Omega)]{u(t)}<C(T),
\end{equation}  
then this solution is global, that is $T_{\max}=\infty$.
\end{theorem}  
This theorem will be proven in Section~\ref{sec:localexistence}.

Next, for a two dimensional domain we present the following boundedness result. 
\begin{theorem}[Global existence]\label{Glo:ex:thm}
Assume \eqref{assumptions} and let $\vbar>0$. Then the 
solution of \eqref{systPE} is globally defined and uniformly bounded in time, that is it satisfies
\begin{equation}\label{unif:bdd}
\sup_{t\in[0,\infty]}\normA[L^\infty(\Omega)]{u(t)}<\infty.
\end{equation}
\end{theorem}

In the proof of global boundedness of solutions of the parabolic-elliptic version of \eqref{systPP:ng} with Robin boundary conditions for $v$ in \cite[p.265]{FLM}, a crucial boundedness assertion results from a study of $\frac{d}{dt} \io u^p$. 
When integrating by parts twice in the integral stemming from the taxis term in order to obtain $Δv$ (so as to plug in the second equation), in \cite{FLM} a boundary term $\intdom u^p(γ-v)g$ arises, which is easily controlled by $\intdom u^p$ or $\io |∇u^{\frac p2}|^2$. In \eqref{systPE}, in contrast, the boundary condition no longer allows to control normal derivatives of $v$, which arise. In Section \ref{sec:boundedness}, we therefore have to use an approach to controlling the taxis term that is less sensitive to the precise boundary conditions.\\
%If we already knew boundedness of $\nabla v$ at this point (e.g. from boundedness of $\io u^p$), we could proceed the same way. But we don't.

Apart from global well-posedness of \eqref{systPE}, we are also interested in the asymptotic behaviour of its solutions. By $(U,V)$ we denote a stationary solution of (\ref{systPE}), i.e. solution of the following elliptic problem
\begin{equation}\label{std:prob}
\left\{ \begin{array}{l}
 0=\nabla\cdot(\nabla U-U\nabla V)+\lambda U-\mu U^2, \\
 \partial_\nu U-U\partial_\nu V|_{\partial\Omega}=0,\\
 -\Delta V=-UV,\\
 V|_{\partial\Omega}=\overline{v}>0. 
 \end{array} \right.
\end{equation}
One may check that $(0,\overline{v})$ is a unique constant non-negative solution of this system, this solution is dynamically unstable. Knosalla and Wróbel in \cite{KW} showed the following result concerning (\ref{std:prob}).
\begin{theorem}
Assume \eqref{ass:domain}.
For $\overline{v}>0$ there exists a positive non-constant classical solution $(U,V)\in\left( C^{2,\sigma}(\overline{\Omega}) \right)^2$ of (\ref{std:prob}). Any such solution satisfies
\begin{equation}\label{U:std:est}
e^{V(x)-\overline{v}}\frac{\lambda}{\mu}\leq U(x)\leq \frac{\lambda}{\mu}e^{V(x)}\quad \text{ for each }x\in\overline{\Omega},
\end{equation}
\begin{equation}\label{V:std:est}
0< V(x)\leq \overline{v}\quad\text{ for each }x\in\overline{\Omega}.
\end{equation}
Moreover there exists a positive constant $v^\star(\lambda,\mu)$, such that for $\overline{v}<v^\star$ the pair $(U,V)$ is a unique positive solution of (\ref{std:prob}).
\end{theorem}

The main result concerning asymptotic behaviour of global solutions of (\ref{systPE}) is stated in the following theorem, which we will prove in Section~\ref{sec:asymptotics}.

\begin{theorem}\label{asymptotics}
Assume \eqref{assumptions} and let $(U,V)$ be a positive stationary solution of (\ref{systPE}). 
There exists a  constant $v^\star_s>0$ such that for every $\overline{v}\in(0,v^\star_s)$ and every global classical solution $(u,v)$ of \eqref{systPE}, 
$$
\norm{U-u(t)}_{L^2(\Omega)}\rightarrow 0,\text{ and }\norm{\nabla(V-v(t))}_{L^2(\Omega)}\rightarrow 0
$$
as time $t$ tends to infinity. The rate of convergence is exponential.
\end{theorem} 

\section{Global existence and boundedness}\label{sec:boundedness}
In this section we show the global existence and global boundedness of solution of (\ref{systPE}) if $\Omega\subset\mathbb{R}^2$. To do that we need to find some {\it a priori} estimates of solutions of this problem. Of course, the main difficulty is to find appropriate estimates on $u$. We start with showing that $u$ has globally bounded $L^1$ norm and locally integrable $L^2$ norm.  %This suffices to obtain global existence of solutions in dimension two. 
Then, applying the  approach taken in \cite{lankeit_smoothing}, we show that in fact the solutions are uniformly bounded in time.    

\begin{lemma}%[$L^1$-, $L^2$-norm estimates of $u$]\label{lem:L1_L2:est}
Assume \eqref{assumptions} and $\vbar>0$. 
Let $(u,v)$ be a classical solution of (\ref{systPE})  defined on $[0,T_{\max})$, then 
\begin{equation}\label{L1_est}
\norm{u(t)}_{L^1(\Omega)}\leq\max\{\norm{u_0}_{L^1(\Omega)},\frac{\lambda|\Omega|}{\mu}\}=:M,\quad \text{ for all } t\in[0,T_{\max} ).
\end{equation}
Moreover, for every $t\in[0,\Tmax)$, 
\begin{equation}\label{uL2:eq}
\norm{u(t)}_{L^2(\Omega)}^2=\frac{1}{\mu}\left(\lambda\norm{u(t)}_{L^1(\Omega)}-\frac{d}{dt}\norm{u(t)}_{L^1(\Omega)}\right).
\end{equation}

\end{lemma}
\begin{proof}
Integration of the $u$-equation leads us to
\begin{align}
\frac{d}{dt}\norm{u(t)}_{L^1(\Omega)}&=\lambda\norm{u(t)}_{L^1(\Omega)}-\mu\int_\Omega u^2(x,t)dx\label{integrateduequation}\\
&\leq \lambda\norm{u(t)}_{L^1(\Omega)}-\frac{\mu}{|\Omega|}\norm{u(t)}_{L^1(\Omega)}^2\nonumber\\
&=\frac{\mu}{|\Omega|}\norm{u(t)}_{L^1(\Omega)}\left(\frac{\lambda|\Omega|}{\mu}-\norm{u(t)}_{L^1(\Omega)}\right),\nonumber
\end{align} 
and this ODE inequality implies (\ref{L1_est}). Integration of \eqref{integrateduequation} then directly shows \eqref{uL2:eq}.
\end{proof}

The following lemma establishes estimates for the second solution component. Here some influence of the size of the boundary data becomes visible.
\begin{lemma}\label{lem:regularityestimatev}
 Assume \eqref{ass:dimension} and \eqref{ass:domain}. Let $p\in[1,\infty)$. Then for every $q\in[1,\frac{2p}{p-2}]$ (if $p<2$) and for every $q\in[1,\infty]$ (if $p>2$) there is $c_1>0$ such that for every $\vbar\in(0,\infty)$ and every nonnegative $u\in C^0(\Ombar)$, every classical solution $v$ of 
 \begin{equation}\label{veq}
  -\Delta v = -uv, \qquad v\big|_{\partial\Omega} = \vbar 
 \end{equation}
satisfies 
\[
 \normA[\Lom q]{\nabla v}\le c_1\vbar \normA[\Lom p]{u}.
\]
Moreover,  there is $c_2>0$ such that for every $\vbar\in(0,\infty)$ and every nonnegative $u\in C^0(\Ombar)$, every classical solution $v$ of \eqref{veq} fulfils 
\begin{equation}\label{GN:ineq}
 \normA[\Lom4]{\nabla v}^4\le c_2\vbar^{4}\normA[\Lom1]{u}\normA[\Lom2]{u}^2.
\end{equation}
\end{lemma}
\begin{proof}
We substitute $z:=\vbar -v$ to rewrite \eqref{veq} into the form 
\begin{equation}\label{zeq}
-\Delta z=(\overline{v}-z)u, \quad z|_{\partial\Omega}=0.
\end{equation}
Using that $\nabla v=\nabla z$ and $0\le v\le \vbar$, along with the embedding $W^{2,p}(\Omega)\hookrightarrow W^{1,q}(\Omega)$ and elliptic regularity, and with a constant $c_1>0$ arising from these, we have 
\[
 \normA[\Lom q]{\nabla v}=\normA[\Lom q]{\nabla z} \le c_1 \normA[\Lom p]{(\vbar-z) u} \le c_1\vbar \normA[\Lom p]{u}.
\]
If we multiply \eqref{zeq} by $z$ and integrate over $\Omega$, then we obtain
\begin{equation}\label{grvest}
\norm{\nabla z(t)}_{L^2(\Omega)}^2=\int_\Omega uz(\overline{v}-z)\leq \overline{v}^2\norm{u(t)}_{L^1(\Omega)}\quad \text{ for all }t\in[0,T_{\max}).  
\end{equation}
Again by elliptic regularity we know that
\begin{equation}\label{ellreg1}
\norm{z(t)}_{W^{2,2}(\Omega)}\leq c_3\norm{u(\overline{v}-z)(t)}_{L^2(\Omega)}\leq c_3 \overline{v}\norm{u(t)}_{L^2(\Omega)},
\end{equation}
for some constant $c_3>0$ depending on $\Omega$ only. Now, with $c_4>0$ taken from the Gagliardo-Nirenberg interpolation inequality, by 
(\ref{grvest}), (\ref{ellreg1}) we have
\begin{align*}
\norm{\nabla z(t)}_{L^4(\Omega)}^4&\leq\left( c_4\norm{\nabla z(t)}_{W^{1,2}(\Omega)}^{\frac{1}{2}}\norm{\nabla z(t)}_{L^2(\Omega)}^{\frac{1}{2}}\right)^4\nonumber\\
&\leq c_4^4\norm{z(t)}_{W^{2,2}(\Omega)}^2\overline{v}^2\normA[\Lom1]{u}\\
&\leq c_2\vbar^4 \norm{u(t)}_{L^2(\Omega)}^2\normA[\Lom1]{u},
\end{align*}
where $c_2=c_3^2c_4^4$.
\end{proof}

Now we shall show that in fact the solution is  uniformly bounded in time. We start with the following simple observation, which will later allow us to find certain times $\hat{t}$ where, e.g., $\io u^2(\hat{t})$ is small.

\begin{lemma}\label{lem:firstbound}
Assume \eqref{assumptions}. 
There is $c_5>0$, independent of $t$, such that for all $\vbar>0$, for the solution $(u,v)$ of \eqref{systPE} and all $t\ge 0$, $τ\in(0,1]$ such that $t+τ\le \Tmax$ we have 
\begin{equation}\label{bd:ul2l2}
\int_t^{t+τ} \int_\Omega u^2 \leq c_5
\end{equation}
and 
\begin{equation}\label{bd:navl4l4}
\int_t^{t+τ}\int_\Omega |\nabla v|^4\leq c_5\vbar^4
\end{equation} 
\end{lemma}
\begin{proof}
Integration, on $(t,t+τ),$ of \eqref{uL2:eq} and \eqref{GN:ineq}  yields
\begin{align*}
\int_t^{t+τ}\normA[L^2(\Omega)]{u(s)}^{2}ds&\leq \frac{1}{\mu}\left(\int_t^{t+τ}\lambda\normA[L^1(\Omega)]{u(s)}ds+\normA[L^1(\Omega)]{u(t)}-\normA[L^1(\Omega)]{u(t+\tau)}\right)\\
&\leq \frac{M(\lambda+1)}{\mu},
\end{align*} 
and 
\begin{align*}
\int_t^{t+τ}\normA[L^4(\Omega)]{\nabla v(s)}^4ds\leq c_{2} \vbar^4M\frac{M(\lambda+1)}{\mu}. 
\end{align*}
Thus it suffices to take $c_5:=\max\{c_2M,1\}\frac{M(\lambda+1)}{\mu}$. 
\end{proof}

For $T\in(0,\infty]$, in the space
$$
V=L^2_{loc}([0,T);W^{1,2}(\Omega))\cap C([0,T);L^2(\Omega))
$$
we define the seminorms
$$
\norm{\varphi}_{V,t_1,t_2}=\sup_{(t_1,t_2)}\norm{\varphi(t)}_{L^2(\Omega)}+\norm{\nabla\varphi}_{L^2(\Omega\times (t_1,t_2))}, \qquad T\ge t_2>t_1\geq 0,\;\; t_2<\infty.
$$
Using the Gagliardo-Nirenberg inequality we can show the following estimate. 
\begin{lemma}\label{lem:normV}
 There is $c_6>0$ such that for any $T\in(0,\infty]$, $t_1,t_2\in [0,T)$ with $t_2>t_1$ we have
 $$
  \norm{\varphi}_{L^4(\Omega\times(t_1,t_2))}^4\leq c_6(1+(t_2-t_1))\norm{\varphi}_{V,t_1,t_2}^4\quad\text{for all }\varphi\in V.
 $$
\end{lemma}
\begin{proof}
 See \cite[Lemma 4.1]{lankeit_smoothing} (or \cite{LSU}).
\end{proof}

\begin{lemma}\label{lem:estimateuVst}
Assume \eqref{assumptions}.
 There is $c_7>0$ such that for every $\vbar>0$, every solution $(u,v)$ of \eqref{systPE} and any $s\ge 0$, $τ\in(0,1]$ and $t\in[s,s+3τ)\subseteq[0,\Tmax)$, 
 $$
  \Vnorm[s,t]{u}^2 \leq c_7+c_7\Vnorm[s,t]{u}^2\left(\int_s^t\int_\Omega |\nabla v|^4\right)^{\frac{1}{2}} + 4 \int_\Omega u^2(\cdot,s)
 $$
\end{lemma}

\begin{proof}
 For the constant $c_8:=|\Omega|\sup_{ξ>0} (2λξ^2-2μξ^3)>0$, we have
 \begin{align*}
  \frac{d}{dt}\int_\Omega u^2 &= 2\int_\Omega u\nabla\cdot(\nabla u-u\nabla v)+2\lambda\int_\Omega u^2-2\mu\int_\Omega u^3\\
  &\le -2\int_\Omega \nabla u\cdot(\nabla u-u\nabla v) + c_8 \\
  &= -2\int_\Omega |\nabla u|^2 + 2\int_\Omega u\nabla u\nabla v+c_8\\
  &\leq -\int_\Omega |\nabla u|^2 + \int_\Omega u^2|\nabla v|^2+c_8.
 \end{align*}
Integration of this inequality on $(s,t)$ and an application of Lemma~\ref{lem:normV} lead us to
 \begin{align*}
  \int_\Omega u^2(\cdot,t)-\int_\Omega u^2(\cdot,s) +\int_s^t\int_\Omega |\nabla u|^2 &\leq \int_s^t \int_\Omega u^2|\nabla v|^2 + c_8(t-s)\\
  &\leq \left(\int_s^t \int_\Omega u^4\right)^{\frac{1}{2}} \left(\int_s^t \int_\Omega |\nabla v|^4\right)^{\frac{1}{2}} + c_8(t-s)\\
  &\leq \sqrt{c_6(1+(t-s))}\Vnorm[s,t]{u}^2\left(\int_s^t \int_\Omega |\nabla v|^4\right)^{\frac{1}{2}} + c_8(t-s)\\
  &\leq \sqrt{4c_6}\Vnorm[s,t]{u}^2\left(\int_s^t \int_\Omega |\nabla v|^4\right)^{\frac{1}{2}} + 3c_8.
\end{align*}
In particular, 
 \begin{align*}
\int_\Omega u^2(\cdot,t) \leq   \sqrt{4c_6}\Vnorm[s,t]{u}^2 \left(\int_s^t \int_\Omega |\nabla v|^4\right)^{\frac{1}{2}} + 3c_8 + \int_\Omega u^2(\cdot,s)
 \end{align*}
and 
 \begin{align*}
\int_s^t \int_\Omega |\nabla u|^2 \leq   \sqrt{4c_6}\Vnorm[s,t]{u}^2\left(\int_s^t \int_\Omega |\nabla v|^4\right)^{\frac{1}{2}} + 3c_8 + \int_\Omega u^2(\cdot,s),
 \end{align*}
hence 
\begin{align*}
 \Vnorm[s,t]{u}^2&\leq 2\sup_{\sigma\in (s,t)} \int_\Omega u^2(\cdot,\sigma ) + 2\normA[L^2(\Omega\times(s,t))]{\nabla u}^2\\
 &\leq 4\left(\sqrt{4c_6}\Vnorm[s,t]{u}^2\left(\int_s^t \int_\Omega |\nabla v|^4\right)^{\frac{1}{2}} + 3c_8 + \int_\Omega u^2(\cdot,s)\right)\\
 &\leq c_7+c_7\Vnorm[s,t]{u}^2\left(\int_s^t \int_\Omega |\nabla v|^4\right)^{\frac{1}{2}}+4\int_\Omega u^2(\cdot,s)
\end{align*}
if $c_7=\max\{8\sqrt{c_6},12c_8\}$.
\end{proof}

\begin{lemma}\label{lem:locallyuniforml2bound}
Assume \eqref{assumptions}. 
There are $c_{9},c_{10}>0$ such that for every $\vbar>0$, every solution $(u,v)$ of \eqref{systPE} and every $t_1\geq 0$, $τ\in(0,1]$ such that $(t_1,t_1+3τ)\subseteq(0,\Tmax)$ we have 
$$
 \int_\Omega u^2(\cdot,t) \leq c_{9}(1+c_{10}^{\vbar^4})\left(1+\int_\Omega u^2(\cdot,t_1)\right) \qquad \forall t\in(t_1,t_1+3τ).
$$
\end{lemma}
\begin{proof}
 Given $t_k<t_1+3τ$, $k\in \mathbb{N}$, we choose $t_{k+1}\in(t_k,t_1+3τ)$ such that 
$$
  \int_{t_k}^{t_{k+1}} \int_\Omega |\nabla v|^4 = \frac{1}{4c_7^2} 
$$
 if such $t_{k+1}$ exists, or $t_{k+1}=t_1+3τ$ otherwise. In any case, we thus obtain 
  \begin{equation}\label{est:nav4small}
  \int_{t_k}^{t_{k+1}} \int_\Omega |\nabla v|^4 \leq \frac{1}{4c_7^2}. 
 \end{equation}
 
 Letting $K:= \min\{k\in \mathbb{N}\mid t_k\ge t_1+3τ\}-1$, we have $t_1,\ldots,t_{K}\in[t_1,t_1+3τ)$ and for all $k\in\{2,\ldots,K\}$, we have 
 $$
  \int_{t_1}^{t_k} \int_\Omega |\nabla v|^4=\frac{k}{4c_7^2}.
 $$
 By \eqref{bd:navl4l4} of Lemma~\ref{lem:firstbound}, 
 $$
  3c_5\vbar^4\geq \int_{t_1}^{t_1+3τ}\int_\Omega |\nabla v|^4 \geq \int_{t_1}^{t_K}\int_\Omega |\nabla v|^4\geq \frac{K}{4c_7^2}
 $$
 so that $K\leq 12c_5c_7^2\vbar^4$.
 
 For $k\in\{1,\ldots, K\}$, according to Lemma~\ref{lem:estimateuVst} and \eqref{est:nav4small},
 \begin{align*}
  \Vnorm[t_k,t_{k+1}]{u}^2&\leq c_7+c_7\Vnorm[t_k,t_{k+1}]{u}^2\left(\int_{t_k}^{t_{k+1}}\int_\Omega |\nabla v|^4\right)^\frac{1}{2} + 4\int_\Omega u^2(\cdot,t_k)\\
  &\leq c_7+\frac{1}{2}\Vnorm[t_k,t_{k+1}]{u}^2+4\int_\Omega u^2(\cdot,t_k)
 \end{align*}
 or, in summary, 
 \begin{align}\label{summaryu2}
  \int_\Omega u^2(\cdot,t_{k+1})\leq \Vnorm[t_k,t_{k+1}]{u}^2 &\leq 2c_7+8\int_\Omega u^2(\cdot,t_k).
 \end{align}
Abbreviating $y_k=\int_\Omega u^2(\cdot,t_k)$, we obtain from $y_{k+1}\leq 2c_7+8y_k$ (for $k\in\{1,\ldots,K\}$) that 
$y_k\leq 8^{k-1}y_1 + 2c_7 \sum_{l=0}^{k-1} 8^l\leq 8^{k-1} y_1 + \frac{2c_7}{7} \cdot 8^k$. 

For any $k\in \{1,\ldots,K\}$, we therefore have 
$$
 \int_\Omega u^2(\cdot,t_k) \leq  \max\{\frac18,\frac{2c_7}7\} 8^K (1+y_1) \le c_{11} c_{10}^{\vbar^4} (1+y_1),
$$
where 
$$
c_{10}=8^{12c_5c_7^2},\qquad c_{11}=\max\{\frac18,\frac{2c_7}7\}.
$$
Therefore, by \eqref{summaryu2},
$$
 \Vnorm[t_k,t_{k+1}]{u}^2\leq 2c_7+8c_{11} c_{10}^{\vbar^4} (1+y_1) \qquad \forall k\in\{1,\ldots,K\}. 
$$
In particular, for every $t\in(t_1,t_1+3τ)$ there is $k$ such that $t\in[t_k,t_{k+1})$ (and $\normA[L^2(\Omega)]{u(\cdot,t)}\leq\Vnorm[t_k,t_{k+1}]{u}$), hence 
$$
 \int_\Omega u^2(\cdot,t) \leq 2c_7+8c_{11} c_{10}^{\vbar^4} (1+y_1) \le c_{9}(1+c_{10}^{\vbar^4})(1+y_1),
$$
with $c_{9}=2c_7+8c_{11}$.
\end{proof}

\begin{lemma}\label{lem:uniformboundul2}
Assume \eqref{assumptions} and let $τ\in(0,1]\cap(0,\Tmax)$.
There exist positive constants $c_{12}(τ)$,  $c_{13}(τ)$ such that for every $\vbar>0$, the solution $(u,v)$ of \eqref{systPE} obeys the following inequalities: 
\begin{equation*}%\label{u:nifest}
   \sup_{t\in[\tau,\Tmax)}\int_\Omega u^2(\cdot,t)\leq c_{12}(1+c_{10}^{\vbar^4})
\end{equation*}
 and
\begin{equation}\label{grv:unifest}
 \sup_{t\in[\tau,\Tmax)}\normA[L^4(\Omega)]{\nabla v(t)}^4\leq c_{13}\vbar^4(1+c_{10}^{\vbar^4}).
\end{equation}

\end{lemma}

\begin{proof}
By \eqref{bd:ul2l2} of Lemma~\ref{lem:firstbound}, for any $n\in \mathbb{N}_0$ there exists $t_n\in[τn,τ(n+1))$ such that  
 $$
 \int_\Omega u^2(\cdot,t_n) \leq \frac{c_5}{τ}.
 $$
According to Lemma~\ref{lem:locallyuniforml2bound}, therefore for every $n\in \mathbb{N}_0$ 
$$
\int_\Omega u^2(\cdot,t)\leq c_{9}(1+c_{10}^{\vbar^4})\left(1+\frac{c_5}{τ}\right) \qquad \forall t\in(t_n,t_n+3τ) \supseteq [τ(n+1),τ(n+2)]
$$
Finally, 
$$
 \bigcup_{n\in\mathbb{N}_0} (t_n,t_n+3τ) \supseteq [τ,+\infty),
$$
and thus we have
$$
\int_\Omega u^2(\cdot,t)\leq c_{9}(1+c_{10}^{\vbar^4})(1+\frac{c_5}{τ})=:c_{12}(1+c_{10}^{\vbar^4})
\qquad \forall t\in[τ,+\infty).
$$
Inequality \eqref{GN:ineq} shows \eqref{grv:unifest} with $c_{13}:=c_2 Mc_{12}$.
\end{proof}

Owing to the fact that the origin of the bounds in Lemma~\ref{lem:uniformboundul2} lies in the spatio-temporal integral estimates of Lemma~\ref{lem:firstbound}, these bounds are temporally local, and do not extend to time intervals of the form $[0,T)$. When the derivation of bounds for higher $L^p$ norms is based on differential inequalities and these estimates, this also means that bounds cannot rely on estimates for the initial data (but instead for $u(τ)$ for some $τ>0$.) Of course, for every single solution, boundedness at such a fixed time is part of the local existence result of Theorem~\ref{locEx}. However, when we deal with a limit of $\vbar$ during the study of long-term behaviour of solutions with small $\vbar$ in Section~\ref{sec:asymptotics}, some uniformity of such bounds in system parameters like $\vbar$ appears desirable. A suitable source is provided by the following elementary statement on ODEs, which relies on the superlinear absorption term and provides a bound independent of initial data or the value of the particular function at a positive time $τ$.

\begin{lemma}\label{lem:ode}
For every $a,b,τ>0$ and $θ>1$ there is $c_{14}(a,b,τ,θ)>0$ such that for every $T\in(0,∞]$ every solution $y\in C^0([0,T))\cap C^1((0,T))$   of 
\begin{equation}\label{ode}
y'+ay^{θ}\le b
\end{equation}
satisfies $y(t)\le c_{14}$ for all $t\in (τ,T)$.
\end{lemma}
\begin{proof}
 Due to $y_*=(\frac{b}{a})^{\frac1{θ}}$ being a supersolution of \eqref{ode}, we know that if $y(t_*)>y_*$ for some (maximal) $t_*\in(0,T)$, then $y(t)>y_*$ for all $t\in(0,t_*)$. Assuming the latter case, we have (from nonpositivity of $b-ay^{θ}$) and by separation of variables 
 \[
  \int_{y(0)}^{y(t)} \frac{1}{b-ay^{θ}} dy \ge t\qquad \forall t\in(0,t_*)
 \]
 and thus 
 \[
  \int_{y(t)}^{\infty} \frac{1}{ay^{θ}-b} dy \ge τ \qquad \forall t\in (τ,T).
 \]
 Since $\lim_{z\to \infty} \int_z^{\infty} \frac{1}{ay^{θ}-b}dy=0$ due to $θ>1$, there is $z^*=z^*(a,b,τ,θ)>0$ such that for all $z>z^*$, $\int_z^\infty\frac{1}{ay^{θ}-b} dy<τ$. The choice of $c_{14}:= \max\{(\frac{b}{a})^{\frac{1}{θ}},z^*(a,b,τ,θ)\}$ then serves to prove that $y(t)\le c_{14}$ for all $t\in (τ,T)$.
\end{proof}

In preparation for the most critical term in the treatment of $\frac{d}{dt} \normA[p]{u(t)}^p$, we establish the following consequence of the Hölder and Gagliardo-Nirenberg inequalities. 
\begin{lemma}
Let $p\in[1,\infty)$. Then there is $c_{15}>0$ such that for every $\vbar>0$ the solution $(u,v)$ of \eqref{systPE} satisfies 
\begin{equation}\label{u:Lp:drift:est}
\int_\Omega u^{p}|\nabla v|^2 dx\leq \frac{4}{p^2}\norm{\nabla u^{\frac{p}{2}}}_{L^2(\Omega)}^2+\left(\frac{(c_{15}p)^2}{16}\norm{u^{\frac{p}{2}}}_{L^2(\Omega)}^2+1\right)\int_\Omega |\nabla v|^4dx+\frac{c_{15}^2}{4}M^4
\end{equation}
in $(0,\Tmax)$.
\end{lemma}
\begin{proof}
By the H{\"o}lder inequality we obtain
\begin{equation}\label{upest}
\int_\Omega u^{p}|\nabla v|^2 dx\leq \left(\int_\Omega u^{2p}dx\right)^{\frac{1}{2}}\left(\int_\Omega |\nabla v|^4dx\right)^{\frac{1}{2}}. 
\end{equation}
We use again the Gagliardo-Nirenberg inequality to estimate the first term of the right hand side of the preceding inequality, namely
\begin{align*}
\left(\int_\Omega u^{2p}dx\right)^{\frac{1}{2}}=\norm{u^{\frac{p}{2}}}_{L^4(\Omega)}^2&\leq\left(c_4\norm{u^{\frac{p}{2}}}_{W^{1,2}(\Omega)}^{\frac{1}{2}}\norm{u^{\frac{p}{2}}}_{L^2(\Omega)}^{\frac{1}{2}}\right)^2\\
&=c_4^2\left(\norm{\nabla u^{\frac{p}{2}}}_{L^2(\Omega)}\norm{u^{\frac{p}{2}}}_{L^2(\Omega)}+\norm{u^{\frac{p}{2}}}_{L^2(\Omega)}^2\right)
\end{align*}
Now we proceed like Winkler in \cite{Winkler1}. By the Hölder inequality, with $\theta=\frac{2(p-1)}{2p-1}$, we have
$$
\norm{u^{\frac{p}{2}}}_{L^2(\Omega)}^2\leq \norm{u^{\frac{p}{2}}}_{L^4(\Omega)}^{2\theta}\norm{u^{\frac{p}{2}}}_{L^{\frac{2}{p}}(\Omega)}^{2(1-\theta)}, 
$$  
and for every $\varepsilon>0$ the Young inequality yields $c_{16}(\varepsilon)>0$ such that 
$$
\norm{u^{\frac{p}{2}}}_{L^2(\Omega)}^2\leq \varepsilon \norm{u^{\frac{p}{2}}}_{L^4(\Omega)}^{2}+c_{16}(\varepsilon)\norm{u^{\frac{p}{2}}}_{L^{\frac{2}{p}}(\Omega)}^{2}.
$$
Choosing $\varepsilon$ in such a way that $\varepsilon c_4^2<1$ leads us to
\begin{equation*}%\label{inequ2}
\left(\int_\Omega u^{2p}dx\right)^{\frac{1}{2}}=\norm{u^{\frac{p}{2}}}_{L^4(\Omega)}^2 \leq
c_{15}\left(\norm{\nabla u^{\frac{p}{2}}}_{L^2(\Omega)}\norm{u^{\frac{p}{2}}}_{L^2(\Omega)}+\norm{u}_{L^1(\Omega)}^2\right)
\end{equation*}
where $c_{15}=\max\{c_4^2,c_{16}(\varepsilon)\}\frac{1}{1-\varepsilon c_4^2}$.
Thus from (\ref{upest}) we get \eqref{u:Lp:drift:est}.
\end{proof}

With this, we can derive bounds of $\normA[\Lom p]{u(t)}$ for larger $p$, which are uniform in $t\in[τ,\Tmax)$ and in $\vbar\in (0,\vhat)$.
\begin{lemma}\label{u:Lp:unif}
Assume \eqref{assumptions}, let $τ>0$ and $\vhat>0$. 
For any  $p\in(1,\infty)$ there exists a constant $c_{17}(p,τ,\vhat)$ such that for every $\vbar\in(0,\vhat)$, the solution $(u,v)$ of \eqref{systPE} satisfies 
\begin{equation}\label{u:uniform:Lp:bound}
 \sup_{t\in[τ,\Tmax)} \normA[L^p(\Omega)]{u(t)}\leq c_{17}(p,τ,\vhat).
\end{equation}
\end{lemma}

\begin{proof}
Multiplying the $u$-equation by $pu^{p-1}$ and integrating over $\Omega$ shows that on $(0,\Tmax)$,
\begin{equation}\label{u:Lp:glob}
\begin{split}
\frac{d}{dt}\norm{u(t)}_p^p+\frac{2(p-1)}{p}\norm{\nabla u^{\frac{p}{2}}(t)}_2^2 
& \leq \frac{p(p-1)}{2}\int_\Omega u^{p}|\nabla v|^2 dx\\ 
&\quad +\lambda p \norm{u(t)}_{L^p(\Omega)}^p-\mu p \norm{u(t)}_{L^{p+1}(\Omega)}^{p+1}.
\end{split}
\end{equation}
Using \eqref{u:Lp:drift:est}, \eqref{grv:unifest}, on $[\frac{τ}2,\Tmax)$ we have 
\begin{align*}
\int_\Omega u^{p}|\nabla v|^2 dx %&\leq \frac{4}{p^2}\norm{\nabla u^{\frac{p}{2}}}_{L^2(\Omega)}^2+\left(\frac{(c_5p)^2}{16}\norm{u^{\frac{p}{2}}}_{L^2(\Omega)}^2+1\right){C_2^*}^2+\frac{c_5^2}{4}M^4\\
&\leq \frac{4}{p^2}\norm{\nabla u^{\frac{p}{2}}}_{L^2(\Omega)}^2+c_{18}\norm{u}_{L^p(\Omega)}^p+c_{19},
\end{align*}
where $c_{18}=\frac{(c_{15}p)^2}{16} c_{13}\vhat^4 (1+c_{10}^{\vhat^4})$ and $c_{19}=c_{13}\vhat^4(1+c_{10}^{\vhat^4})+\frac{c_{15}^2}{4}M^4$ with $c_{13}=c_{13}(\frac{τ}2)$ from Lemma~\ref{lem:uniformboundul2}. 

Since the Young inequality yields $c_{20}>0$ such that on $(\frac{τ}2,\Tmax)$ we have
$$
-\frac{\mu p}{2}\norm{u}_{L^{p+1}(\Omega)}^{p+1}\leq -(c_{18}\frac{p(p-1)}{2}+\lambda p)\norm{u}_{L^p(\Omega)}^p+c_{20}
$$
and, by Hölder's inequality, similarly
\[
 -\frac {μp}{2}\normA[\Lom{p+1}]{u(t)}^{p+1}\le -\frac{μp}2 |\Omega|^{-\frac1p} \normA[\Lom p]{u(t)}^{p(1+\frac1p)},
\]
from (\ref{u:Lp:glob}) we get   
$$
\frac{d}{dt}\norm{u(t)}_{L^p(\Omega)}^p+\frac{μp}2|\Omega|^{-\frac1p}(\norm{u(t)}_{L^p(\Omega)}^p)^{1+\frac1p}\leq c_{21}:=\frac{p(p-1)c_{19}}2+c_{20}.
$$
on $(\frac{τ}2,\Tmax)$. 
This inequality assures the global uniform boundedness of the $L^p$-norm of solutions as in \eqref{u:uniform:Lp:bound}, with $c_{17}(p,τ,\vhat)=c_{14}(\frac{μp}{2}|\Omega|^{-\frac1p},c_{21},\frac{τ}2,1+\frac1p)$ taken from Lemma~\ref{lem:ode}.
\end{proof}

\begin{lemma}\label{u:Linfty}
 Assume \eqref{assumptions} and $\vbar>0$. Then for the solution $(u,v)$ of \eqref{systPE}, 
 \[
  \sup_{t\in[0,\Tmax)} \normA[\Lom{\infty}]{u(t)}<\infty
 \]
 and, in particular, $\Tmax=\infty$.
\end{lemma}
\begin{proof}
 We let $p>4$ and taking $τ=\frac{\Tmax}2>0$ and using that, according to Theorem~\ref{locEx}, \mbox{$\sup_{t\in[0,τ]}\normA[\Lom p]{u(t)}<\infty$,} from Lemma~\ref{u:Lp:unif} we obtain that $\sup_{t\in[0,\Tmax)}\normA[\Lom p]{u(t)}<\infty$, and, due to Lemma~\ref{lem:regularityestimatev}, $\sup_{t\in[0,\Tmax)}\normA[\Lom p]{f(t)}<\infty$ for $f(t)=u(t)\nabla v(t)$.
 These bounds turn into uniform $\Lom{\infty}$-bounds by a Moser-Alikakos iteration procedure \cite{Alikakos1, Alikakos2}; more precisely in the present situation we apply \cite[Lemma A.1]{taowin_quasilin}. (Note that this lemma is applicable despite the conditions $\partial_\nu u\le 0$ and $f\cdot \nu\le 0$ on $\partial \Omega$ not being satisfied; for the boundary terms to vanish in the last equation on \cite[p.710]{taowin_quasilin}, it is sufficient if $\partial_\nu u+f\cdot\nu\le 0$, which here is the case.) 
 That $\Tmax=\infty$ then immediately follows from the extensibility criterion \eqref{cond:globalex} in Theorem~\ref{locEx}.
\end{proof}

\begin{proof}[Proof of Theorem \ref{Glo:ex:thm}]
That the solutions whose local existence was asserted by Theorem~\ref{locEx} actually are global and bounded has been shown with Lemma~\ref{u:Linfty}.
\end{proof}

\section{Asymptotic behaviour of global solutions}\label{sec:asymptotics}

In this section we proceed to the question of asymptotic behaviour of solutions. We introduce the notation: $\widetilde{u}=u-U$, $\widetilde{v}=v-V$. The difference $\widetilde{u}$ satisfies the equation
\begin{equation}\label{u:differ}
\left\{ \begin{array}{l}
 \widetilde{u}_t=\nabla\cdot(\nabla \widetilde{u}-(\widetilde{u}\nabla v+U\nabla\widetilde{v}))+\mu\widetilde{u}\left(\frac{\lambda}{\mu}-u-U\right), \\
 \partial_\nu \widetilde{u}-(\widetilde{u}\partial_\nu v+U\partial_\nu \widetilde{v})|_{\partial\Omega}=0, 
 \end{array} \right.
\end{equation} 
and the difference $\widetilde{v}$ fulfils 

\begin{equation}\label{v:differ}
\left\{ \begin{array}{l}
 \Delta\widetilde{v}=u\widetilde{v}+V\widetilde{u}, \\
 \widetilde{v}|_{\partial\Omega}=0, 
 \end{array} \right..
\end{equation}
\begin{lemma}
Assume \eqref{assumptions} and $\vbar>0$. The functions $(u,v)$, $(U,V)$, $(\tilde{u},\tilde{v})$ as introduced earlier then satisfy
\begin{align}\label{u:differ:finalest}
\frac{d}{dt} \normA[\Lom2]{\utilde(t)}^2&+\normA[\Lom2]{∇\utilde(t)}+2μ\io \utilde^2(t)\left(u(t)+U - \frac{λ}{μ} - \frac{1}{μ} \normA[\Lom{\infty}]{∇v(t)}^2\right)\nonumber\\
&\le \frac{2λ^2}{μ^2} e^{2\vbar}\left( \frac{\vbar^2}{4ε_1}\normA[\Lom2]{\utilde}^2- \io (u-ε_1)\vtilde^2 \right)
\end{align}
and 
\begin{equation}\label{v:differ:est}
\norm{\nabla \widetilde{v}(t)}_{L^2(\Omega)}^2+\int_\Omega (u(t)-\varepsilon_1)\widetilde{v}(t)^2dx\leq \frac{1}{4\varepsilon_1}\norm{V}_\infty^2\norm{\widetilde{u}(t)}_{L^2(\Omega)}^2
\end{equation}
for every $ε_1>0$ and every $t>0$.
\end{lemma}

\begin{proof}
We multiply equation (\ref{u:differ}) by $2\widetilde{u}$ and integrate over $\Omega$, then by the Young inequality we obtain
\begin{align*}
\frac{d}{dt}\norm{\widetilde{u}(t)}_{L^2(\Omega)}^2&+2\norm{\nabla\widetilde{u}(t)}_{L^2(\Omega)}^2+2\mu\int_\Omega\widetilde{u}(t)^2\left(u(t)+U-\frac{\lambda}{\mu}\right)dx\\
&=2\int_\Omega \left(\widetilde{u}(t)\nabla v(t)+ U\nabla \widetilde{v}(t)\right)\cdot \nabla \widetilde{u}(t)dx\\
&\leq \norm{\nabla\widetilde{u}(t)}_{L^2(\Omega)}^2+\norm{\widetilde{u}(t)\nabla v(t)+ U\nabla \widetilde{v}(t)}_{L^2(\Omega)}^2\\
& \leq \norm{\nabla\widetilde{u}(t)}_{L^2(\Omega)}^2+2\left(\norm{\widetilde{u}(t)\nabla v(t)}_{L^2(\Omega)}^2+ \norm{U\nabla \widetilde{v}(t)}_{L^2(\Omega)}^2\right)\\
& \leq \norm{\nabla\widetilde{u}(t)}_{L^2(\Omega)}^2+2\left(\norm{\widetilde{u}(t)}_{L^2(\Omega)}^2\norm{\nabla v(t)}_\infty^2+ \norm{U}_\infty^2\norm{\nabla \widetilde{v}(t)}_{L^2(\Omega)}^2\right),
\end{align*}
which implies 
 \begin{equation}\label{u:differ:est}
\begin{split}
\frac{d}{dt}\norm{\widetilde{u}(t)}_{L^2(\Omega)}^2+\norm{\nabla\widetilde{u}(t)}_{L^2(\Omega)}^2&+2\mu\int_\Omega\widetilde{u}(t)^2\left(u(t)+U-\frac{\lambda}{\mu}-\frac{1}{\mu}\norm{\nabla v(t)}_\infty^2\right)dx\\
&\leq 2\norm{U}_\infty^2\norm{\nabla \widetilde{v}(t)}_{L^2(\Omega)}^2\leq \frac{2λ^2}{μ^2} e^{2V}\norm{\nabla \widetilde{v}(t)}_{L^2(\Omega)}^2
\end{split}
\end{equation}
due to the estimate $U\le\frac{λ}{μ}e^V$ from \eqref{U:std:est}.
Now, we multiply (\ref{v:differ}) by $\widetilde{v}$ and integrate,
\begin{align*}
\norm{\nabla \widetilde{v}(t)}_{L^2(\Omega)}^2+\int_\Omega u(t)\widetilde{v}(t)^2dx&=\int_{\Omega}V\widetilde{u}(t)\widetilde{v}(t)dx\\
&\leq \varepsilon_1\norm{\widetilde{v}(t)}_{L^2(\Omega)}^2+\frac{1}{4\varepsilon_1}\norm{V}_\infty^2\norm{\widetilde{u}(t)}_{L^2(\Omega)}^2.
\end{align*}
Thus we obtain \eqref{v:differ:est}. If we combine \eqref{u:differ:est} and \eqref{v:differ:est} and the estimate $V\le \vbar$ from \eqref{V:std:est}, we arrive at \eqref{u:differ:finalest}.
\end{proof}

By comparison we are able to find appropriate sub-solution of \eqref{systPE} which will be a crucial lower estimate for $u$ in our further analysis.
\begin{lemma}\label{u:lower:sol}
Let $u$ be a classical solution of (\ref{systPE})    defined on $\overline{\Omega}\times [0,T]$ with $\inf_{x\in\Omega} u_0(x)>0$. The solution of the following ODE problem
\begin{equation}\label{y:lo:est}
\left\{ \begin{array}{l}
y'(t)=\lambda y(t)-(\mu+\overline{v})y(t)^2\\
y(0)=\inf_{x\in\Omega}u_0(x)
 \end{array} \right.
\end{equation} 
is a sub-solution of \eqref{systPE}, i.e.
$$
u(x,t)\geq y(t)
$$
for all $x\in\overline{\Omega}$ and $t\in[0,T]$.
\end{lemma}  
\begin{proof}
 Consider an auxiliary initial value problem
\begin{equation*}
\left\{ \begin{array}{l}
y'_\varepsilon(t)=\lambda y_\varepsilon(t)-(\mu+\overline{v})y_\varepsilon(t)^2\\
y_\varepsilon(0)=y(0)-\varepsilon
 \end{array} \right.
\end{equation*} 
with $\varepsilon$ from the interval $I=(0,y(0))$.
We show that $u(x,t)>y_\varepsilon(t)$ for any $\varepsilon\in I$, and this will imply the desired conclusion. Consider the difference $w(x,t):=u(x,t)-y_\varepsilon(t)$ and suppose that there exists a subset of $\overline{\Omega}\times(0,T]$ where we have $w\leq 0$. By  continuity of $w$ and the fact that $w(\cdot,0)>0$  there exists a point $(x_0,t_0)\in\overline{\Omega}\times(0,T]$ at which we have $w(x_0,t_0)=0$, and $w(x,t)>0$ for all $t\in[0,t_0)$ and all $x\in \Omega$.   The function $w(\cdot,t_0)$ achieves its minimum at the point $x_0$. Assuming that $x_0\in\partial\Omega$ we should have $\partial_\nu w(x_0,t_0)\leq 0$, but   
$$
\partial_\nu w(x_0,t_0)=\partial_\nu u(x_0,t_0)=u(x_0,t_0)\partial_\nu v(x_0,t_0)=y_\varepsilon(t_0)\partial_\nu v(x_0,t_0)>0
$$ 
according to Hopf's boundary point lemma, \cite[Lemma~3.4]{GiT}, a contradiction. Now assume that $x_0\in\Omega$. Then $\Delta w (x_0,t_0)\geq 0$, $\nabla w(x_0, t_0)=0$ and thus at $(x_0,t_0)$ we have 
\begin{align*}
0\geq w_t&= \Delta u-\nabla u\cdot \nabla v-u\Delta v+\lambda u-\mu u^2-\lambda y_\varepsilon +(\mu+\overline{v})y_\varepsilon ^2\\
&=\Delta u-\nabla u\cdot \nabla v-(v+\mu) u^2 +(\mu+\overline{v})y_\varepsilon ^2\\
&\geq (\overline{v}-v)y_\varepsilon^2,
\end{align*}
a contradiction, since Theorem~\ref{locEx} asserts strict positivity of $\vbar-v$.
\end{proof}
\begin{remark}
It follows from the above result that the constant steady state 
$(0,\overline{v})$ is always unstable.

\end{remark}

We are ready now to prove the main result of this section.
\begin{proof}[Proof of Theorem \ref{asymptotics}]
We let $τ>0$ and $\vhat>0$ and fix $ε_1\in(0,\min\{\inf_{\Omega} u_0,\frac{λ}{μ+\vhat}\})$. Then by Lemma \ref{u:lower:sol}, for every $\vbar\in(0,\vhat)$, the solution $(u,v)$ of \eqref{systPE} satisfies $u(t)\ge \min\{ε_1,\frac{λ}{μ+\vbar}\}=ε_1$ for all $t\in(0,∞)$.
From Lemma~\ref{lem:regularityestimatev} and Lemma~\ref{u:Lp:unif}, we conclude that 
\[
 \normA[\Lom\infty]{∇v(t)}\le c_1\vbar\normA[\Lom3]{u(t)}\le c_1\vbar c_{17}(3,τ,\vhat)\quad\text{for all } t\in(τ,\infty).
\]
On account of \eqref{u:differ:finalest} we then have 
\[
\frac{d}{dt} \normA[\Lom2]{\utilde(t)}^2+2μ\io \utilde^2(t)\left(u(t)+U - \frac{λ}{μ} - \frac{1}{μ} c_1^2\vbar^2 (c_{17}(3,τ,\vhat))^2 - \frac{λ^2}{μ^3} e^{2\vbar} \frac{\vbar^2}{4ε_1} \right)\le 0,
\]
that is, 
\begin{equation}\label{u:differ:est1}
\frac{d}{dt}\norm{\widetilde{u}(t)}_{L^2(\Omega)}^2+2\mu\mathcal{F}(t,\overline{v})\norm{\widetilde{u}(t)}_{L^2(\Omega)}^2\leq 0
\end{equation}
for every $t>τ$, where, with $y$ from \eqref{y:lo:est} and $c_{22}:=c_1^2c_{17}^2(3,τ,\vhat)$,
 $$
 \mathcal{F}(t,\overline{v})=\left(y(t)+e^{-\overline{v}}\frac{\lambda}{\mu}-\frac{\lambda}{\mu}-\overline{v}^2\frac{c_{22}}{\mu}-\overline{v}^2\frac{\lambda^2e^{2\overline{v}}}{4\varepsilon_1\mu^3} \right).
 $$
By a standard ODE result we have $y(t)\geq \min\{\inf_{x\in\Omega}u_0(x),\frac{\lambda}{\mu+\overline{v}}\}$ for any $t\geq 0$, and thus
$$
\mathcal{F}(t,\overline{v})\geq \mathcal{F}(\overline{v})=\left(\min\left\{\inf_{x\in\Omega}u_0(x),\frac{\lambda}{\mu+\overline{v}}\right\}+e^{-\overline{v}}\frac{\lambda}{\mu}-\frac{\lambda}{\mu}-\overline{v}^2\frac{c_{22}}{\mu}-\overline{v}^2\frac{\lambda^2e^{2\overline{v}}}{4\varepsilon_1\mu^3}\right).
$$
It is clear that $\lim_{\overline{v}\searrow 0}\mathcal{F}(\overline{v})=\min\left\{\inf_{x\in\Omega}u_0(x),\frac{\lambda}{\mu}\right\}>0$, and so there exists a positive constant $v^\star_s(\lambda,\mu,3,\vhat,τ)$, such that for $\overline{v}<v^\star_s$ we have $\mathcal{F}(\overline{v})>0$. Whenever $\overline{v}\in(0,v^\star_s)$, from (\ref{u:differ:est1}) we obtain  
$$
\frac{d}{dt}\norm{\widetilde{u}(t)}_{L^2(\Omega)}^2+2\mu\mathcal{F}(\overline{v})\norm{\widetilde{u}(t)}_{L^2(\Omega)}^2\leq 0
$$
for every $t>τ$, 
and this yields that
$$
\norm{\widetilde{u}(t)}_{L^2(\Omega)}^2\leq \norm{\widetilde{u}(τ)}_{L^2(\Omega)}^2 e^{-2\mu\mathcal{F}(\overline{v})(t-τ)}
$$
for all $t>τ$. Exponential convergence of $\normA[\Lom2]{\nabla\vtilde}$ thus follows from \eqref{v:differ:est}.
\end{proof}

\begin{remark}
 Due to the uniform boundedness and by classical regularity theory, the proven $\Lom2$-convergence can easily be seen to imply convergence in stronger topologies, cf. e.g. the proof of \cite[Lemma 5.4]{LankeitWang}.
\end{remark}

\section{Local existence}\label{sec:localexistence}

In this section our goal is to prove that the problem \eqref{systPE} is well posed at least locally in time in any space dimension. We use here approach due to Amann \cite{Amann1}, a similar approach were used by Delgado et al.\ in \cite{Delgado}. For our purpose, we consider a modification of problem \eqref{systPE},
\begin{equation}\label{systPE:mod}
\left\{ \begin{array}{l}
 u_t=\nabla\cdot(\nabla u-u\nabla v)+\lambda u-\mu u^2, \\
 \partial_\nu u|_{\partial\Omega}=u\partial_\nu v|_{\partial\Omega},\\
 \Delta v=v(u)_+,\\
 v|_{\partial\Omega}=\overline{v}>0.\\
 u(x,0)=u_0(x), 
 \end{array} \right.
\end{equation}
where $(u)_+=\max\{0,u\}$ is the positive part of $u$.
The main result of this section is the following.  

\begin{theorem}\label{systPE:mod:ex}
Let $u_0(x)\in W^{1,p}(\Omega)$ for $p>n$, then there exists a unique classical solution $(u,v)$ of \eqref{systPE:mod}, 
$$
u\in C([0,T_{\max} ); W^{1,p}(\Omega))\cap C^{2,1}(\overline{\Omega}\times(0,T_{\max})),
$$
$$
v\in C([0,T_{\max});C^{2,\sigma_*}(\overline{\Omega}))\cap C^{0,1}((0,T_{\max});C^{2,\sigma_*}(\overline{\Omega})) 
$$
 for some $\sigma_*\in(0,1)$, where $T_{\max}\in(0,\infty]$ is the maximal time of existence of this solution. If for every $T>0$ there exists a constant $C(T)$ such that
 \begin{equation}\label{Extend:est2}
\sup_{t\in[0,T]\cap[0,T_{\max})}\normA[L^{\infty}(\Omega)]{u(t)}<C(T)
\end{equation}  
then $T_{\max}=+\infty.$
\end{theorem}

\begin{proof}[Proof of Theorem \ref{locEx}]
Theorem \ref{locEx} follows from Theorem \ref{systPE:mod:ex}, Lemma \ref{u:lower:sol} and Lemma \ref{linElest}.
\end{proof}

  We begin with some preliminaries. It is well known that the operator $(\mathcal{A},\mathcal{B})$, where
$$
\mathcal{A}z:=-\Delta z +z,\quad \mathcal{B}z:=\partial_\nu z|_{\partial \Omega},
$$
is a particular example of normally elliptic operator in the sense of Amann \cite{Amann1} and it is a generator of an analytic semigroup in $L^p(\Omega)$. We define 
$$
W^{s,p}_{\mathcal{B}}(\Omega):=\left\{\begin{array}{lr}
\{z\in W^{s,p}(\Omega),\,\, \mathcal{B}z=0\},&  1+1/p<s\leq 2,\\ 
W^{s,p}(\Omega),& -1+1/p<s<1+1/p,\\
(W^{-s,p'}(\Omega))',& -2+1/p<s\leq -1+1/p,\\
\end{array}\right.
$$
and by $A_{\alpha-1}$ denote the $W^{2\alpha-2,p}_{\mathcal{B}}(\Omega)$-realization of $(\mathcal{A},\mathcal{B})$. Due to \cite[Theorem 7.1]{Amann1} $(W^{2\alpha-2,p}_{\mathcal{B}}(\Omega),A_{\alpha-1})$ is an element of the interpolation-extrapolation scale generated by the $L^p$-realization of $(\mathcal{A},\mathcal{B})$. Of course,  $A_{\alpha-1}\in\mathcal{L}(W^{2\alpha,p}_{\mathcal{B}}(\Omega),W^{2\alpha-2,p}_{\mathcal{B}}(\Omega))$ and the operator $A_{\alpha-1}$ is a generator of the analytic semigroup $e^{-tA_{\alpha -1}}$ in $W^{2\alpha-2,p}_{\mathcal{B}}(\Omega)$. In the following lemma we recall some important properties of this semigroup.   

\begin{lemma}\label{lem:5.2}
Assume that $p>1$ and $1<\gamma<2\alpha<1+1/p$. Then there exist $\kappa(\gamma)\in(0,1)$ and $c_{23}>0$ such that the estimate 
\begin{equation}\label{heatest}
\norm{e^{-tA_{\alpha-1}}f}_{W^{\gamma,p}(\Omega)}\leq c_{23}\frac{1}{t^{\kappa}}e^{-\delta t}\norm{f}_{W_{\mathcal{B}}^{2\alpha-2,p}(\Omega)}
\end{equation}
holds for any $t>0$,  $f\in W^{2\alpha-2,p}_{\mathcal{B}}(\Omega)$ and with $\delta\in(0,1)$. Moreover, there exists a constant $c_{24}$ such that 
\begin{equation}\label{bound:op}
\norm{e^{-tA_{\alpha-1}}f}_{W^{1,p}(\Omega)}\leq c_{24} \norm{f}_{W^{1,p}(\Omega)}, \text{ for every } t\in[0,\infty) \text{ and every } f\in W^{1,p}(\Omega).
\end{equation}
\end{lemma}

\begin{proof}
As pointed out above, $(W^{2\alpha-2}_{\mathcal{B}}(\Omega),A_{\alpha-1})$ is an element of the interpolation-extrapolation scale generated by the $L^p$-realization of $(\mathcal{A},\mathcal{B})$, thus by \cite[Theorem V.2.1.3]{Amann2} the restriction of $e^{-tA_{\alpha-1}}$ to $W^{1,p}(\Omega)$ is an analytic semigroup and by \cite[Theorem 1.3.4]{Henry} we obtain \eqref{bound:op}.\\
By \cite[Theorem 7.2]{Amann1} we have $W^{\gamma,p}(\Omega)=W^{\gamma,p}_{\mathcal{B}}(\Omega)=(W^{2\alpha-2,p}_{\mathcal{B}}(\Omega),W^{2\alpha ,p}_{\mathcal{B}}(\Omega))_{\kappa, p}$, where $\gamma=\kappa 2\alpha+(1-\kappa)(2\alpha-2)$, and so $\kappa(\gamma):=\gamma/2-\alpha+1\in(0,1)$. Since $A_{\alpha-1}$ is a generator of an analytic semigroup in $W^{2\alpha-2}_{\mathcal{B}}(\Omega)$, by the interpolation inequality we obtain
$$
\normA[W^{\gamma,p}(\Omega)]{e^{-tA_{\alpha-1}}f}\leq c_{25}\normA[W^{2\alpha-2,p}_{\mathcal{B}}(\Omega)]{A_{\alpha-1}e^{-tA_{\alpha-1}}f}^\kappa \normA[W^{2\alpha-2,p}_{\mathcal{B}}(\Omega)]{e^{-tA_{\alpha-1}}f}^{1-\kappa}
$$ 
and \eqref{heatest} follows from \cite[Theorem 1.3.4]{Henry}.
\end{proof}
We will also need the next lemma, which is useful in showing higher regularity of solutions of certain abstract Cauchy problems.

\begin{lemma}\label{high:reg}
Assume that $2\alpha\in(1,1+1/p),$ $p>n,$ and let $u\in C([0,T_{\max});W^{1,p}(\Omega))\cap C^1((0,T_{\max}),W^{2\alpha-2,p}_{\mathcal{B}}(\Omega))$ be a solution of
$$
u_t+A_{\alpha-1}u=F(u),\quad u(0)=u_0. 
$$
with any locally Lipschitz continuous function $F:U\subset W^{1,p}(\Omega)\rightarrow W^{2\alpha-2,p}_{\mathcal{B}}(\Omega)$, such that for every compact $K\subset U$ there is $L_K>0$ such that 
\begin{equation}\label{lipestF}
\normA[W^{2\alpha-2,p}_{\mathcal{B}}(\Omega)]{F(u_1)-F(u_2)}\leq L_{K} \normA[W^{1,p}(\Omega)]{u_1-u_2} \quad\text{ for all } u_1,u_2\in K.
\end{equation}
Then $u\in C^1((0,T_{\max}),W^{\gamma,p}(\Omega))$ with $\gamma\in(1,2\alpha).$
\end{lemma}
\begin{proof}
Since $A_{\alpha-1}$ is a sectorial operator we can define its fractional powers $(A_{\alpha-1})^\theta$ for $\theta\in (0,1)$ with the domains $D((A_{\alpha-1})^\theta)$. Conjunction of \cite[Prop. 2.2.15]{Lunardi}, \cite[Prop. 1.2.3]{Lunardi} and \cite[Thm. 7.2]{Amann1} yields
\begin{equation}\label{fract:imbed}
D((A_{\alpha-1})^\theta)\hookrightarrow (W^{2\alpha-2,p}_\mathcal{B}(\Omega),W^{2\alpha,p}_\mathcal{B}(\Omega))_{\kappa(\gamma),p}=W^{\gamma,p}(\Omega)
\end{equation} 
for $\theta>\kappa(\gamma)$. Using this and \eqref{lipestF} we see that $F:D((A_{\alpha-1})^\theta)\rightarrow W^{2\alpha-2,p}_\mathcal{B}(\Omega)$ is locally Lipschitz continuous and by \cite[Thm. 3.5.2]{Henry} we obtain that $u\in C^{1}((0,T_{\max}),D((A_{\alpha-1})^\theta))$, and again by \eqref{fract:imbed}, $u\in C^1((0,T_{\max}),W^{\gamma,p}(\Omega))$. 
\end{proof}

 To deal with nonhomogeneous boundary conditions we introduce the boundary extension operator $\mathcal{B}^c$.    
For a given $g\in C([0,T],W^{1-1/p,p}(\partial\Omega))$ by $\mathcal{B}^c g$ we denote a solution of the following elliptic problem 
\begin{equation}\label{aux:bdop}
\mathcal{A}z=0,\quad \mathcal{B}z=g(t).
\end{equation} 
By the results of Grisvard (see \cite[Thm. 2.4.2.7]{Grisvard}) there exists a unique $W^{2,p}(\Omega)$ solution of (\ref{aux:bdop}), and $\norm{z}_{W^{2,p}(\Omega)}\leq c_{26} \norm{g(t)}_{W^{1-1/p,p}(\partial\Omega)}$. Thus, since $W^{2,p}(\Omega)\hookrightarrow W^{2\alpha,p}_{\mathcal{B}}(\Omega)$ for $2\alpha<1+1/p$, we see that $\mathcal{B}^c\in \mathcal{L}(W^{1-1/p,p}(\partial\Omega),W^{2\alpha,p}_{\mathcal{B}}(\Omega))$ and finally we obtain that $A_{\alpha-1}\mathcal{B}^c\in \mathcal{L}(W^{1-1/p,p}(\partial\Omega),W^{2\alpha-2,p}_{\mathcal{B}}(\Omega))$ for $1/p<2\alpha <1+1/p$ with
\begin{equation}\label{bdop:est}
\norm{A_{\alpha-1}\mathcal{B}^c g(t)}_{W^{2\alpha-2,p}_{\mathcal{B}}(\Omega))}\leq c_{27} \norm{g(t)}_{W^{1-1/p,p}(\partial\Omega)}.
\end{equation}

Now we present some useful lemma concerning the $v$-equation.
\begin{lemma}\label{linElest}
Assume that $p>n$, then for any given function $\hat{u}\in C([0,T],W^{1,p}(\Omega))$ and number $\vbar>0$, the elliptic problem 
\begin{equation}\label{linE}
\left\{ \begin{array}{l}
 \Delta v(t)=v(t)(\hat{u}(t))_+,\\
 v|_{\partial\Omega}=\overline{v}>0, 
 \end{array} \right.
\end{equation}
has a unique solution $v\in C([0,T],C^{2,\sigma_*}(\overline{\Omega}))$ for some $\sigma_*\in (0,1)$ with 
\begin{equation}\label{linE:es}
\norm{v(t)}_{C^{2,\sigma_*}(\overline{\Omega})}\leq \overline{v}(1+c_{28}\norm{\hat{u}(t)}_{W^{1,p}(\Omega)})
\end{equation}
for $t\in[0,T]$, and moreover $0<v(t)\leq\overline{v}$ for all $t\in[0,T]$ and $v(t)<\vbar$ in $\Omega$ whenever $u_+(t)\not\equiv 0$. If $\hat{u}\in C^1((0,T);W^{1,p}(\Omega))$, then $v\in C^{0,1}((0,T);C^{2,\sigma_*}(\overline{\Omega}))$ 
\end{lemma}
\begin{proof}
It is clear that positive part $(x)_+=\frac{x+|x|}{2}$ is a Lipschitz function. By the Sobolev embedding theorem there exists some $\sigma_1(p,n)\in(0,1)$ such that for each $t\in[0,T]$ we have $\hat{u}(t)\in W^{1,p}(\Omega)\hookrightarrow C^{\sigma_1}(\overline{\Omega})$, and so $(\hat{u}(t))_+\in C^{\sigma_1}(\overline{\Omega})$. The equation is linear, thus by the Schauder estimates there exists a unique solution $v(t)\in C^{2,\sigma_*}(\overline{\Omega})$, with $\sigma_*=\min\{\sigma,\sigma_1\}$. Using the strong maximum principle \cite[Thm.3.5]{GiT}, we also have that $0<v(t)\leq \overline{v}$ in $\Ombar$. If $(u(t))_+\not\equiv 0$, moreover, $v(t)<\vbar$ in $\Omega$.
. If  $t_1,t_2\in[0,T]$, then the difference $v(t_1)-v(t_2)$ satisfies the equation
\begin{align*}
\Delta (v(t_1)-v(t_2))-(\hat{u}(t_1))_+(v(t_1)-v(t_2))=v(t_2)((\hat{u}(t_1))_+-(\hat{u}(t_2))_+)
\end{align*}
with 
$$
(v(t_1)-v(t_2))|_{\partial\Omega}=0,
$$
and using Schauder's estimates and Sobolev's embedding we have
\begin{align*}
\norm{v(t_1)-v(t_2)}_{C^{2,\sigma_*}(\overline{\Omega})}&\leq c_{29}\norm{v(t_2)(\hat{u}(t_1)-\hat{u}(t_2))}_{C^{\sigma_*}(\overline{\Omega})}\\
&\leq c_{30}\norm{v(t_2)}_{C^{\sigma_*}(\overline{\Omega})} \norm{(\hat{u}(t_1)-\hat{u}(t_2))}_{C^{\sigma_*}(\overline{\Omega})}\\
&\leq c_{31}\norm{v(t_2)}_{C^{\sigma_*}(\overline{\Omega})} \norm{(\hat{u}(t_1)-\hat{u}(t_2))}_{W^{1,p}(\Omega)},
\end{align*}
and the continuity follows. Now, by the substitution $z=\overline{v}-v$ we obtain the equation
$$
-\Delta z +(\hat{u})_+z=\overline{v}(\hat{u})_+,\quad z|_{\partial\Omega}=0.
$$ 
By the Schauder estimates and Sobolev embedding we have
$$
\normA[C^{2,\sigma_*}(\overline{\Omega})]{z}\leq c_{28}\vbar\normA[W^{1,p}(\Omega)]{\hat{u}},
$$
using this we obtain \eqref{linE:es}.
Assume now that $\hat{u}\in C^1((0,T);W^{1,p}(\Omega))$ and $t_1,t_2\in(0,T)$ using the above estimates, we obtain the other part of the statement.
\end{proof}

\begin{proof}[Proof of Theorem \ref{systPE:mod:ex}]
We are going to use the Banach fixed point theorem. For $T>0$ and a given function $\hat{u}\in C([0,T],W^{1,p}(\Omega))$ we consider the linear problem  
\begin{equation}\label{linPP}
\left\{ \begin{array}{l}
 u_t+\mathcal{A}u=F(\hat{u}(t),v(t)), \\
 \mathcal{B} u=g(t),\\
 u(0)=u_0
  \end{array} \right.
\end{equation}
where
$$
F(\hat{u}(t),v(t))=-\nabla \hat{u}\cdot\nabla v+(\lambda+1) \hat{u}-(\mu+v) \hat{u}^2,
$$
$$
g(t)=\hat{u}\partial_\nu v|_{\partial\Omega}.
$$
and $v(t)$ solves the elliptic problem (\ref{linE}).  
For $2\alpha\in(1,1+1/p)$, by the generalized variations of constant formula \cite[(11.20)]{Amann1}, we can rewrite problem (\ref{linPP}) into  the following  integral equation
$$
u(t)=e^{-tA_{\alpha-1}}u_0+\int_0^te^{-(t-\tau)A_{\alpha-1}}\left(F(\hat{u}(\tau),v(\tau))+A_{\alpha-1}\mathcal{B}^cg(\tau)\right)d\tau,
$$
thus we define the following operator
$$
\mathcal{T}[\hat{u}](t):=e^{-tA_{\alpha-1}}u_0+\int_0^te^{-(t-\tau)A_{\alpha-1}}\left(F(\hat{u}(\tau),v(\tau))+A_{\alpha-1}\mathcal{B}^cg(\tau)\right)d\tau
$$
for $\hat{u}\in B_T^R=\{f\in C([0,T],W^{1,p}(\Omega)):\,\sup_{t\in[0,T]}\norm{f(t)}_{W^{1,p}(\Omega)}\leq R\}$, with $R>2c_{24}\norm{u_0}_{W^{1,p}(\Omega)}$. Our goal is to show that $\mathcal{T}$ is a contraction in $B_T^R$ with sufficiently small $T$. For $i=1,2$ denote by $v_i(t)$  the solution of (\ref{linE}) with given $\hat{u}_i(t)$, and let $g_i(t)=\hat{u}_i(t)\partial_\nu v_i(t)|_{\partial\Omega}$.  Using the embedding  $L^p(\Omega)\hookrightarrow W^{2\alpha-2,p}_\mathcal{B}(\Omega)$ and Lemma \ref{linElest} we have
\begin{align*}
\norm{\nabla v_1\cdot\nabla\hat{u}_1-\nabla v_2\cdot\nabla\hat{u}_2}_{W^{2\alpha-2,p}_{\mathcal{B}}(\Omega)}&\leq c_{32}\norm{\nabla v_1\cdot\nabla\hat{u}_1-\nabla v_2\cdot\nabla\hat{u}_2}_{L^p(\Omega)}\\
&\leq c_{33} \left(\norm{v_1}_{C^{2,\sigma_*}(\overline{\Omega})}\norm{\hat{u}_1-\hat{u}_2}_{W^{1,p}(\Omega)}+\norm{ \hat{u}_2}_{W^{1,p}(\Omega)}\norm{ v_1- v_2}_{C^{2,\sigma_*}(\overline{\Omega})}\right)\\
&\leq c_{34}(R)\norm{\hat{u}_1-\hat{u}_2}_{W^{1,p}(\Omega)},
\end{align*}
\begin{align*}
\norm{(\lambda+1)(\hat{u}_1-\hat{u}_2)}_{W^{2\alpha-2,p}_{\mathcal{B}}(\Omega)}\leq c_{35}(R) \norm{\hat{u}_1-\hat{u}_2}_{W^{1,p}(\Omega)}, 
\end{align*}
and similarly, using the assumption that $p>n$,
\begin{align*}
\norm{(\mu+v_1)\hat{u}_1^2-(\mu+v_2)\hat{u}_2^2}_{W^{2\alpha-2,p}_{\mathcal{B}}(\Omega)}&\leq c_{36}\left((\mu+\norm{v_1}_\infty)\norm{(\hat{u}_1-\hat{u}_2)(\hat{u}_1+\hat{u}_2)}_{L^p(\Omega)}+\norm{\hat{u}_2^2(v_1-v_2)}_{L^p(\Omega)}\right)\\
&\leq c_{37}\left(\norm{\hat{u}_1-\hat{u}_2}_{W^{1,p}(\Omega)}\norm{\hat{u}_1+\hat{u}_2}_{\infty}+\norm{\hat{u}_2}_{\infty}^2\norm{(v_1-v_2)}_{C^{2,\sigma_*}(\overline{\Omega})}\right)\\
&\leq c_{38}(R)\norm{\hat{u}_1-\hat{u}_2}_{W^{1,p}(\Omega)}.
\end{align*}
Gathering the above estimates leads us to
\begin{equation}\label{N:lest}
\norm{F(\hat{u}_1(t),v_1(t))-F(\hat{u}_2(t),v_2(t))}_{W^{2\alpha-2,p}_{\mathcal{B}}(\Omega)} \leq c_{39}(R)\norm{\hat{u}_1(t)-\hat{u}_2(t)}_{W^{1,p}(\Omega)}.
\end{equation} 
 
By (\ref{bdop:est}) and the trace theorem we obtain 
\begin{align*}
\norm{A_{\alpha-1}\mathcal{B}^c(\hat{u}_1\partial_\nu v_1-\hat{u}_2\partial_\nu v_2)}_{W^{2\alpha-2,p}_{\mathcal{B}}(\Omega)}&\leq c_{40}\left(\norm{(\hat{u}_1-\hat{u}_2)\partial_\nu v_1}_{W^{1-1/p,p}(\partial\Omega)}+\norm{\hat{u}_2\partial_\nu( v_1- v_2)}_{W^{1-1/p,p}(\partial\Omega)}\right)\\
&\leq c_{41}\left(\norm{(\hat{u}_1-\hat{u}_2)|\nabla v_1|}_{W^{1,p}(\Omega)}+\norm{\hat{u}_2|\nabla  (v_1- v_2)|}_{W^{1,p}(\Omega)}\right)\\
&\leq c_{42}\left(\norm{\hat{u}_1-\hat{u}_2}_{W^{1,p}(\Omega)}\norm{v_1}_{C^{2,\sigma_*}(\overline{\Omega})}+\norm{\hat{u}_2}_{W^{1,p}(\Omega)}\norm{v_1-v_2}_{C^{2,\sigma_*}(\overline{\Omega})}\right)\\
&\leq c_{43}(R)\norm{\hat{u}_1-\hat{u}_2}_{W^{1,p}(\Omega)},  
\end{align*}
thus
\begin{equation}\label{BDO:est}
\norm{A_{\alpha-1}\mathcal{B}^cg_1(t)-A_{\alpha-1}\mathcal{B}^cg_2(t)}_{W^{2\alpha-2,p}_{\mathcal{B}}(\Omega)}\leq c_{43}(R)\norm{\hat{u}_1(t)-\hat{u}_2(t)}_{W^{1,p}(\Omega)}
\end{equation}

Using the embedding $W^{\beta,p}(\Omega)\hookrightarrow W^{1,p}(\Omega)$ for $\beta>1$, Lemma \ref{lem:5.2}, (\ref{N:lest}), and (\ref{BDO:est}) (with $\hat{u}_2(t)\equiv 0$) we obtain with $\delta\in(0,1)$ 
\begin{align*}
\norm{\mathcal{T}[\hat{u}](t)}_{W^{1,p}(\Omega)}
&\leq c_{24}\norm{u_0}_{W^{1,p}(\Omega)}+ c_{44}\int_0^t\norm{e^{-(t-\tau)A_{\alpha-1}}\left(F(\hat{u}(\tau),v(\tau))+A_{\alpha-1}\mathcal{B}^cg(\tau)\right)}_{W^{\beta,p}(\Omega)}d\tau\\
&\leq \frac{R}{2}+ c_{45}\int_0^te^{-(t-\tau)\delta}\frac{1}{(t-\tau)^\kappa}\left(\norm{F(\hat{u}(\tau),v(\tau))}_{W^{2\alpha-2,p}_{\mathcal{B}}(\Omega)}+\norm{A_{\alpha-1}\mathcal{B}^cg(\tau)}_{W^{2\alpha-2,p}_{\mathcal{B}}(\Omega)}\right)d\tau\\
&\leq \frac{R}{2}+c_{46}(R)\frac{1}{1-\kappa}T^{1-\kappa},
\end{align*}
and so $\mathcal{T}[B_T^R]\subset B_T^R$ for sufficiently small $T$. 
Proceeding similarly as above we finish with 
\begin{align*}
\norm{\mathcal{T}[\hat{u}_1](t)-\mathcal{T}[\hat{u}_2](t)}_{W^{1,p}(\Omega)}&\leq c_{47}(R)\frac{1}{1-\kappa}T^{1-\kappa}\norm{\hat{u}_1(t)-\hat{u}_2(t)}_{W^{1,p}(\Omega)}
\end{align*}
and thus we see that $\mathcal{T}$ is a contraction when $T$ is small. We found a unique fixed point $u\in B_T^R$ of the integral operator $\mathcal{T}$, and by standard argumentation this fixed point is, in fact, unique in $C([0,T];W^{1,p}(\Omega))$. Continuing our above fixed point argument we obtain a solution defined on the maximal time interval $[0,T_{\max})$, and either $T_{\max}=+\infty$ or  
\begin{equation}\label{Extend:est1}
\lim_{t\nearrow T_{\max}}\normA[W^{1,p}(\Omega)]{u(t)}=+\infty,\text{ if } T_{\max}<+\infty.
\end{equation}
In fact (by \cite{Henry} or \cite{Lunardi}) $u(t)$ is a classical solution, that is we have $u\in C([0,T_{\max});W^{1,p}(\Omega))\cap  C^1((0,T_{\max}),W^{2\alpha-2,p}_{\mathcal{B}}(\Omega))$ and $u$ is the unique solution of the following Cauchy problem 

\begin{equation*}%\label{abs:cauchy}
\left\{\begin{array}{l}
u_t+A_{\alpha-1}u=F(u,v)+A_{\alpha-1}\mathcal{B}^c\left( u\partial_\nu v|_{\partial\Omega}\right),\quad t \in(0,T_{\max})\\
u(0)=u_0.
\end{array}\right. 
\end{equation*}
It follows from inequalities \eqref{N:lest},  \eqref{BDO:est} and Lemma \ref{high:reg} that $u\in  C^1((0,T_{max}),W^{\gamma,p}(\Omega))$ with $\gamma\in(1,2\alpha)$. In particular we have $u\in  C^1((0,T_{max}),W^{1,p}(\Omega))$, and
by Lemma \ref{linElest}, $v(t)\in C^{0,1}((0,T_{\max}),C^{2,\sigma_*}(\overline{\Omega}))$. This regularity is sufficient to show that $u$  in fact is a classical solution. Indeed, let $T_{\max}>T>\delta'>0$, and $\psi\in C^{\infty}(\mathbb{R})$ be a cut-off function such that $\psi(t)=0$ for $t\leq\delta'/2$, $\psi(t)=1$ for $t\geq\delta'$. The function $w=\psi(t) u$ solves the Cauchy problem
\begin{equation}\label{aux:cauchy}
w_t+A_{\alpha-1}w+\nabla w\cdot\nabla v=\tilde{F}+A_{\alpha-1}\mathcal{B}^c\left( w\partial_\nu v|_{\partial\Omega}\right),\quad w(0)=0
\end{equation}
where $\tilde{F}=\psi \left(F(u,v)+\nabla u\cdot\nabla v\right)+\psi_t u \in C^{\sigma',\sigma'/2}(\overline{\Omega}\times[0,T])$ for some $\sigma'\in(0,1).$  Notice that the linear problem 
\begin{equation}\label{aux:Weq}
\left\{\begin{array}{l}
W_t-\Delta W-\nabla v\cdot\nabla W+W=\tilde{F}\\
\partial_\nu W- W\partial_\nu v|_{\partial\Omega}=0\\
W(x,0)=0
\end{array}\right.
\end{equation}
has a unique solution $W\in C^{2+\sigma'',1+\sigma''/2}(\overline{\Omega}\times[0,T])$
by \cite[Corollary 5.1.22]{Lunardi}, which is also a solution of \eqref{aux:cauchy}. Since \eqref{aux:cauchy} has a unique solution, we have that $w=W,$ and thus
$$
u\in C^{2,1}(\overline{\Omega}\times (0,T_{\max})).
$$

Now we show that condition \eqref{Extend:est2} is sufficient to show that $T_{\max}=+\infty$.
Let $W=\psi(t)u$ be as in the above considerations. We know that $W$ is a classical, and thus $W^{2,1}_p(\Omega\times(0,T))$ solution of \eqref{aux:Weq} for any $p\in(1,+\infty)$. By \cite[IV Theorem 9.1]{LSU} we know that 
$$
\normA[W^{2,1}_p(\Omega\times (\delta',T))]{u}\leq \normA[W^{2,1}_p(\Omega\times (0,T))]{W}\leq c_{48}\normA[L^p(\Omega\times (0,T))]{\tilde{F}}\leq c_{49} \sup_{t\in[0,T]}\normA[L^\infty(\Omega)]{u(t)},
$$ 
and from \cite[Lemma II.3.3]{LSU}, for $p>n+2$, we have 
$$
\sup_{t\in[\delta',T]}\normA[W^{1,p}(\Omega)]{u(t)}\leq c_{50}\normA[{C^{1+\sigma''',\frac{1+\sigma'''}{2}}(\overline{\Omega}\times[\delta',T])}]{u}\leq c_{51}\normA[W^{2,1}_p(\Omega\times (\delta',T))]{u}.
$$
We see here that \eqref{Extend:est2} excludes situation \eqref{Extend:est1}.
\end{proof}

\footnotesize
\def\cprime{$'$}

\end{document}